\documentclass[11pt]{amsart}
\usepackage{microtype}
\usepackage{amssymb}
\usepackage{amsthm}
\usepackage{amscd}
\usepackage{hyperref}

\newtheorem{theorem}{Theorem}[section]
\newtheorem{lemma}[theorem]{Lemma}

\theoremstyle{definition}
\newtheorem{definition}[theorem]{Definition}

\newcommand{\BigO}[1]{\ensuremath{\operatorname{O}\bigl(#1\bigr)}}

\makeatletter

\def\dotminussym#1#2{%
  \setbox0=\hbox{$\m@th#1-$}%
  \kern.5\wd0%
  \hbox to 0pt{\hss\hbox{$\m@th#1-$}\hss}%
  \raise.6\ht0\hbox to 0pt{\hss$\m@th#1.$\hss}%
  \kern.5\wd0}

\mathchardef\mhyphen="2D

\allowdisplaybreaks[2]

\begin{document}

\title{The Regularity Lemma with bounded VC Dimension}
\author{Henry Towsner}
\date{\today}
\thanks{Partially supported by NSF grant DMS-1340666.}
\address {Department of Mathematics, University of Pennsylvania, 209 South 33rd Street, Philadelphia, PA 19104-6395, USA}
\email{htowsner@math.upenn.edu}
\urladdr{\url{www.math.upenn.edu/~htowsner}}


\maketitle

\section{Introduction}

In this note we give a proof of Szemer\'edi's celebrated regularity lemma \cite{MR540024} in the special case where the graph has bounded VC dimension.  In the general case it is known that the tower exponential bounds given by Szemer\'edi's proof are essentially optimal \cite{MR1445389}, but in this special case we obtain doubly exponential bounds.  After placing this online, we learned that a stronger result had already been obtained by Lov\'asz and Szegedy \cite{MR2815610}, giving polynomial bounds (and a stronger condition on the partition).  We leave this note online since the proof may still be of interest.

The results here are reminiscent of recent results by Malliaris and Shelah \cite{2011arXiv1102.3904M} obtaining improved bounds for the regularity lemma under various model theoretic assumptions (including bounded VC dimension, under its model theoretic name, NIP \cite{MR1171563}).  However their results focus on eliminating or controlling irregular pairs, while the result here keeps the irregular pairs but requires fewer components in the partition.

\section{$\epsilon$-Approximations}

Throughout this note, we will be concerned with large finite sets $X$ and $Y$.  We will use measure-theoretic notation for the normalized counting measure: when $X'\subseteq X$, $\mu(X')=\frac{|X'|}{|X|}$, when $Y'\subseteq Y$, $\mu(Y')=\frac{|Y'|}{|Y|}$, and when $E\subseteq X\times Y$, $\mu(E)=\frac{|E|}{|X|\cdot |Y|}$.

When $E\subseteq X\times Y$, for $x\in X$, we write $E_x=\{y\mid (x,y)\in E\}$ and for $y\in Y$ we write $E^y=\{x\mid (x,y)\in E\}$.



\begin{definition}
Let $E\subseteq X\times Y$.  If $I\subseteq Y$, we say $\{E_x\}_{x\in X}$ \emph{shatters} $I$ if for each $J\subseteq I$, there is an $x\in X$ with $E_x\cap I=J$.

The \emph{VC dimension} of a collection $\{E_x\}_{x\in X}$ is the supremum of $|I|$ for those $I\subseteq Y$ such that $\{E_x\}_{x\in X}$ shatters $I$.  The \emph{dual VC dimension} of $\{E_x\}_{x\in X}$ is the VC dimension of $\{E^y\}_{y\in Y}$.

When $E\subseteq X\times Y$, the VC dimension of $E$ is the larger of the VC dimension of $\{E_x\}_{x\in X}$ and the dual VC dimension of $\{E_x\}_{x\in X}$.
\end{definition}
An equivalent definition of the dual VC dimension is the supremum of $|I|$ where $I\subseteq X$ and for each $J\subseteq I$, there is a $y\in Y$ such that $y\in E_x$ iff $x\in J$.  The definition of VC dimension in terms of the collection $\{E_x\}_{x\in X}$ is the standard one, but for us it will generally be more convenient to view VC dimension as a property of the set of pairs $E$ (and we are only interested in the symmetric case where we look at the larger of the VC and dual VC dimensions).

Recall the following standard properties of VC dimension:
\begin{lemma}
\begin{enumerate}
\item If $\{E_x\}_{x\in X}$ has VC dimension $d$, the dual VC dimension of $\{E_x\}_{x\in X}$ is less than $2^{d+1}$,
\item If $X'\subseteq X$ and $Y'\subseteq Y$ then $E\cap(X'\times Y')$ has VC dimension no larger than the VC dimension of $E$,
\item {}\emph{[Shelah-Sauer \cite{MR0307902,MR0307903} ]} If $\{E_x\}_{x\in X}$ has VC dimension $d$ then for any set $I\subseteq Y$ with $|I|\geq d$, 
\[|\{E_x\cap I\mid x\in X\}|=|\{J\subseteq I\mid \exists x\in X\ J=E_x\cap I\}|\leq \left(\frac{e|I|}{d}\right)^d.\]
\end{enumerate}
\end{lemma}

\begin{theorem}[$\epsilon$-net Theorem \cite{MR884223,MR1139078}]
If $\{E_x\}_{x\in X}$ has VC dimension at most $d$, $d\geq 2$, then for sufficiently large $r\geq 2$ there is a set $\hat Y\subseteq Y$ such that $|\hat Y|\leq \BigO{dr\ln r}$ and for each $x\in X$ with $\mu(E_x)\geq 1/r$, $E_x\cap \hat Y$ is non-empty.
\end{theorem}

\begin{definition}
  We say $\hat Y\subseteq Y$ is an \emph{$\epsilon$-net for differences} if whenever $E_x\cap \hat Y=E_{x'}\cap \hat Y$, $\mu(E_x\bigtriangleup E_{x'})<\epsilon$.
\end{definition}

\begin{lemma}\label{thm:diff_net}
  If $\{E_x\}_{x\in X}$ has VC dimension at most $d$, $d\geq 2$, then for each $r\geq 2$ there is an $\epsilon$-net for differences $\hat Y\subseteq Y$ such that $|\hat Y|\leq \BigO{dr\ln r}$.
\end{lemma}
\begin{proof}
  For $x,x'\in X$, write $E_{x-x'}=E_x\setminus E_{x'}$.  By \cite{MR2797943,MR876079}, the VC dimension of $\{E_{x-x'}\}_{x,x'\in X}$ is bounded by $10d$, so by the $\epsilon$-net theorem, there is a set $\hat Y\subseteq Y$ such that $|\hat Y|\leq \BigO{dr\ln 2r}=\BigO{dr\ln r}$ so that whenever $\mu(E_{x-x'})\geq 1/2r$, $E_{x-x'}\cap \hat Y$ is non-empty.  In particular, if $E_x\cap \hat Y=E_{x'}\cap \hat Y$ then
\[\mu(E_x\bigtriangleup E_{x'})=\mu(E_x\setminus E_{x'})+\mu(E_{x'}\setminus E_x)\leq 1/2r+1/2r=1/r.\]
\end{proof}


\section{Regularity}

In this section we fix a set $E\subseteq X\times Y$.  We write $d(X',Y')=\frac{\mu(E\cap(X',Y'))}{\mu(X')\mu(Y')}$.

Throughout this section, a partition $\mathcal{P}$ is a pair $(\{X_i\}_{i\leq n},\{Y_j\}_{j\leq m})$ where $\{X_i\}_{i\leq n}$ is a partition of $X$ and $\{Y_j\}_{j\leq m}$ is a partition of $Y$.  We set $|\mathcal{P}|=\max\{n,m\}$.

\begin{definition}
We say $\mathcal{P}'=(\{X'_i\},\{Y'_j\})$ \emph{refines} $\mathcal{P}=(\{X_i\},\{Y_j\})$ if for each $X_i$ and each $X'_{i'}$, either $X'_{i'}\subseteq X_i$ or $X'_{i'}\cap X_i=\emptyset$, and if for each $Y_j$ and each $Y'_{j'}$, either $Y'_{j'}\subseteq Y_j$ or $Y'_{j'}\cap Y_j=\emptyset$.
\end{definition}
Equivalently, $\mathcal{P}'$ refines $\mathcal{P}$ if for each $X_i$, there is some $I'$ such that $\{X_{i'}\}_{i'\in I'}$ is a partition of $X_i$, and similarly for each $Y_j$.  Clearly this relation is symmetric and transitive.

\begin{definition}
  We define
\[\rho(\mathcal{P})=\sum_{i\leq n,j\leq m}d^2(X_i,Y_j)\mu(X_i)\mu(Y_j).\]
\end{definition}

We recall the following standard facts about $\rho$:
\begin{lemma}
  \begin{enumerate}
  \item $0\leq \rho(\mathcal{P})\leq 1$,
  \item If $\mathcal{P}'$ refines $\mathcal{P}$ then $\rho(\mathcal{P})\leq\rho(\mathcal{P}')$.
  \end{enumerate}
\end{lemma}

\begin{definition}
  A pair $(X_i,Y_j)$ is \emph{$\epsilon$-regular} if whenever $X'\subseteq X_i$, $Y'\subseteq Y_j$ with $\mu(X')\geq\epsilon\mu(X_i)$ and $\mu(Y')\geq\epsilon\mu(Y_j)$, we have
\[\left|d(X_i,Y_j)-d(X',Y')\right|<\epsilon.\]

If $\mathcal{P}$ is a partition, we write $\mathfrak{I}(\epsilon,\mathcal{P})$ for the set of $(i,j)$ such that $(X_i,Y_j)$ is \emph{not} $\epsilon$-regular.

We say $\mathcal{P}$ is $\epsilon$-regular if $\sum_{(i,j)\in\mathfrak{I}(\epsilon,\mathcal{P})}\mu(X_i)\mu(Y_j)<\epsilon$.
\end{definition}

\begin{definition}
  Let $\hat X\subseteq X, \hat Y\subseteq Y$ be finite sets.  The partition \emph{induced by $\hat X,\hat Y$} takes, for each $I\subseteq \hat Y$, $X_I=\{x\mid E_x\cap \hat Y=I\}$ and for $J\subseteq \hat X$, $Y_J=\{y\mid E^y\cap \hat X=J\}$.
\end{definition}

\begin{lemma}
  Suppose $\hat X,\hat Y$ are $\epsilon$-nets for differences and let $\mathcal{P}=(\{X_i\},\{Y_j\})$ be the partition induced by $\hat X,\hat Y$.  Then whenever $x,x'\in X_i$, we have $\mu(E_x\bigtriangleup E_{x'})<\epsilon$ and whenever $y,y'\in Y_j$, we have $\mu(E^y\bigtriangleup E^{y'})<\epsilon$.
\end{lemma}

\begin{lemma}
  Let $\hat X\subseteq \hat X'\subseteq X, \hat Y\subseteq \hat Y'\subseteq Y$ be given.  Then the partition induced by $\hat X',\hat Y'$ refines the partition induced by $\hat X,\hat Y$.
\end{lemma}

\begin{lemma}
Suppose $\mathcal{P}$ is the partition induced by $\hat X,\hat Y$ and is not $1/r$-regular.  Further, suppose $E$ has VC dimension at most $d$.  Then there are $\hat X'\supseteq\hat X$ and $\hat Y'\supseteq\hat Y$ such that the partion $\mathcal{P}'$ induced by $\hat X',\hat Y'$ satisfies:
  \begin{enumerate}
  \item $|\mathcal{P}'|\leq \BigO{\left(|\mathcal{P}|dr^3\ln r^3\right)^d}$,
  \item $\rho(\mathcal{P}')\geq\rho(\mathcal{P})+1/10^3r^7$.
  \end{enumerate}
\end{lemma}
\begin{proof}
Let $\mathcal{P}$ be the partition $(\{X_i\}_{i\leq n},\{Y_j\}_{j\leq m})$.  For each $i$, define a measure $\mu_i$ on subsets of $X_i$ by $\mu_i(X')=\frac{\mu(X')}{\mu(X_i)}=\frac{|X'|}{|X_i|}$.  Similarly, define a measure $\mu^j$ on subsets of $Y_j$ by $\mu^j(Y')=\frac{\mu(Y')}{\mu(Y_j)}=\frac{|Y'|}{|Y_j|}$.  Finally, define a measure $\mu_i^j$ on subsets of $X_i\times Y_j$ by $\mu_i^j(S)=\frac{\mu(S)}{\mu(X_i)\mu(Y_j)}$.

For each $i\leq n$, let $\hat X'_i$ be a $1/10r^3$-net for differences in $X_i$ with respect to the measure $\mu_i$; take $\hat X'=\hat X\cup\bigcup_{i\leq n}\hat X'_i$.  Similarly, for each $j\leq m$, let $\hat Y'_j$ be a $1/10r^3$-net for differences in $Y_j$ with respect to the measure $\mu^j$.  By Lemma \ref{thm:diff_net}, each $\hat X'_i$ and $\hat Y'_j$ may be taken to have size at most $\BigO{dr^3\ln r^3}$.  This means $|\hat X'|,|\hat Y'|\leq \BigO{|\mathcal{P}| dr^3\ln r^3}$.  Let $\mathcal{P}'$ be the partition induced by $\hat X',\hat Y'$.  By Shelah-Sauer, $|\mathcal{P}'|\leq \BigO{\left(|\mathcal{P}| dr^3\ln r^3\right)^d}$.

It remains to show that
\[\rho(\mathcal{P}')\geq\rho(\mathcal{P})+1/10^3r^7.\]
For each $i,j$, $\mathcal{P}'$ induces a partition $\mathcal{P}'_{i,j}$ of $X_i,Y_j$, and we have
\[\rho(\mathcal{P}')=\sum_{i\leq n,j\leq m}\rho_{i,j}(\mathcal{P}'_{i,j})\mu(X_i)\mu(Y_j)\]
where
\[\rho_{i,j}(\mathcal{P}'_{i,j})=\sum_{X'_{i'}\subseteq X_i,Y'_{j'}\subseteq Y_j}d^2(X'_{i'},Y'_{j'})\mu_i(X'_{i'})\mu^j(Y'_{j'}).\]
So it suffices to show that whenever $(i,j)\in\mathfrak{I}(1/r,\mathcal{P})$,
\[\rho_{i,j}(\mathcal{P}'_{i,j})\geq d^2(X_i,Y_j)+1/10^3r^6.\]

So consider some $(i,j)\in\mathfrak{I}(1/r,\mathcal{P})$, and let $\mathcal{P}'_{i,j}=(\{X'_{i'}\}_{i'\leq n'},\{Y'_{j'}\}_{j'\leq m'})$.  Let $X'\subseteq X_i, Y'\subseteq Y_j$ witness the failure of $1/r$-regularity.  That is, $\mu_i(X')\geq 1/r$, $\mu^j(Y')\geq 1/r$, and
\[\left|d(X',Y')-d(X_i,Y_j)\right|\geq 1/r.\]
Assume $d(X',Y')\geq d(X_i,Y_j)+1/r$ (the case where $d(X',Y')\leq d(X_i,Y_i)-1/r$ is symmetric).  Let $\tilde X'$ consist of those $x\in X_i$ such that $\frac{\mu^j(E_x\cap Y')}{\mu^j(Y')}\geq d(X_i,Y_j)+1/2r$; clearly $\frac{\mu_i(\tilde X')}{\mu_i(X')}\geq 1/2r$, and so $\mu_i(\tilde X')\geq 1/2r^2$.  Set
\[\tilde X=\bigcup\{X'_{i'}\mid X'_{i'}\cap \tilde X'\neq\emptyset\}.\]
Whenever $x\in\tilde X$, we have $x\in X'_{i'}$ and some $x'\in \tilde X'$ so that $\mu_j((E_x\bigtriangleup E_{x'})\cap Y_j)<1/10r^3$.  In particular, since $\frac{\mu^j(E_{x'}\cap Y')}{\mu^j(Y')}\geq d(X_i,Y_j)+1/2r$ and $\mu^j(Y')\geq 1/r$, $\frac{\mu^j(E_x\cap Y')}{\mu^j(Y')}\geq d(X_i,Y_j)+2/5r$.  Note that $\tilde X'\supseteq\tilde X$, and so $\mu_i(\tilde X)\geq 1/2r^2$.

Now let $\tilde Y'$ consist of those $y\in Y_j$ such that $\frac{\mu_i(E^y\cap \tilde X)}{\mu_i(\tilde X)}\geq d(X_i,Y_j)+1/5r$.  Clearly $\mu^j(\tilde Y')\geq 1/5r^2$.  Set
\[\tilde Y=\bigcup\{Y'_{j'}\mid Y'_{j'}\cap \tilde Y'\neq\emptyset\}.\]
Again, whenever $y\in\tilde Y$ we have a $y'$ with $\mu^i((E^y\bigtriangleup E^{y'})\cap X_i)<1/10r^3$ and $\frac{\mu_i(E^y\cap \tilde X)}{\mu_i(\tilde X)}\geq d(X_i,Y_j)+2/5r$, and therefore $\frac{\mu_i(E^y\cap \tilde X)}{\mu_i(\tilde X)}\geq d(X_i,Y_j)+1/10r$.  It follows that $d(\tilde X,\tilde Y)\geq d(X_i,Y_j)+1/10r$.

Consider the partition $\mathcal{P}^*_{i,j}=(\{\tilde X,X_i\setminus\tilde X\},\{\tilde Y,Y_j\setminus\tilde Y\})$.  $\mathcal{P}'_{i,j}$ refines $\mathcal{P}^*_{i,j}$ (since $\tilde X,\tilde Y$ were defined to be unions of components from $\mathcal{P}'_{i,j}$), so it suffices to show that
\[\rho_{i,j}(\mathcal{P}^*_{i,j})\geq d^2(X_i,Y_j)+1/10^3r^6.\]
Since $\mu(\tilde X)\mu(\tilde Y)\geq 1/10r^4$ and $d(\tilde X,\tilde Y)\geq d(X_i,Y_j)+1/10r$, this follows from standard calculations.


\end{proof}

\begin{theorem}
  If $E$ has VC dimension at most $d$, there is a $1/r$-regular partition $\mathcal{P}$ with $|\mathcal{P}|\leq \BigO{\left(dr^3\ln r^3\right)^{d^{2\cdot 10^3r^7}}}$.
\end{theorem}
\begin{proof}
  Let $\mathcal{P}_0$ be the trivial partition $(\{X\},\{Y\})$.  Given $\mathcal{P}_i$ not $1/r$-regular, the previous lemma tells us there is a $\mathcal{P}_{i+1}$ refining $\mathcal{P}_i$ with $|\mathcal{P}_{i+1}|\leq \BigO{\left(|\mathcal{P}_i|dr^3\ln r^3\right)^d}$ and $\rho(\mathcal{P}_{i+1})\geq\rho(\mathcal{P}_i)+1/10^3r^7$.  Since $0\leq \mathcal{P}_0$ and $\mathcal{P}_n\leq 1$, we must have $n\leq 10^3r^7$, and therefore there is an $n\leq 10^3r^7$ with $\mathcal{P}_n$ $1/r$-regular.

It is easy to check inductively that 
\[|\mathcal{P}_i|\leq \BigO{\left(dr^3\ln r^3\right)^{d^{2i}}}.\]
\end{proof}

\bibliographystyle{plain}
\bibliography{../../Bibliographies/main}

\begin{thebibliography}{10}

\bibitem{MR876079}
R.~M. Dudley.
\newblock A course on empirical processes.
\newblock In {\em \'{E}cole d'\'et\'e de probabilit\'es de {S}aint-{F}lour,
  {XII}---1982}, volume 1097 of {\em Lecture Notes in Math.}, pages 1--142.
  Springer, Berlin, 1984.

\bibitem{MR1445389}
W.~T. Gowers.
\newblock Lower bounds of tower type for {S}zemer\'edi's uniformity lemma.
\newblock {\em Geom. Funct. Anal.}, 7(2):322--337, 1997.

\bibitem{MR884223}
David Haussler and Emo Welzl.
\newblock {$\epsilon$}-nets and simplex range queries.
\newblock {\em Discrete Comput. Geom.}, 2(2):127--151, 1987.

\bibitem{MR1139078}
J{\'a}nos Koml{\'o}s, J{\'a}nos Pach, and Gerhard Woeginger.
\newblock Almost tight bounds for {$\epsilon$}-nets.
\newblock {\em Discrete Comput. Geom.}, 7(2):163--173, 1992.

\bibitem{MR1171563}
Michael~C. Laskowski.
\newblock Vapnik-{C}hervonenkis classes of definable sets.
\newblock {\em J. London Math. Soc. (2)}, 45(2):377--384, 1992.

\bibitem{MR2815610}
L{\'a}szl{\'o} Lov{\'a}sz and Bal{\'a}zs Szegedy.
\newblock Regularity partitions and the topology of graphons.
\newblock In {\em An irregular mind}, volume~21 of {\em Bolyai Soc. Math.
  Stud.}, pages 415--446. J\'anos Bolyai Math. Soc., Budapest, 2010.

\bibitem{2011arXiv1102.3904M}
M.~{Malliaris} and S.~{Shelah}.
\newblock {Regularity lemmas for stable graphs}.
\newblock {\em ArXiv e-prints}, February 2011.

\bibitem{MR0307902}
N.~Sauer.
\newblock On the density of families of sets.
\newblock {\em J. Combinatorial Theory Ser. A}, 13:145--147, 1972.

\bibitem{MR0307903}
Saharon Shelah.
\newblock A combinatorial problem; stability and order for models and theories
  in infinitary languages.
\newblock {\em Pacific J. Math.}, 41:247--261, 1972.

\bibitem{MR540024}
Endre Szemer{\'e}di.
\newblock Regular partitions of graphs.
\newblock In {\em Probl\`emes combinatoires et th\'eorie des graphes ({C}olloq.
  {I}nternat. {CNRS}, {U}niv. {O}rsay, {O}rsay, 1976)}, volume 260 of {\em
  Colloq. Internat. CNRS}, pages 399--401. CNRS, Paris, 1978.

\bibitem{MR2797943}
Aad van~der Vaart and Jon~A. Wellner.
\newblock A note on bounds for {VC} dimensions.
\newblock In {\em High dimensional probability {V}: the {L}uminy volume},
  volume~5 of {\em Inst. Math. Stat. Collect.}, pages 103--107. Inst. Math.
  Statist., Beachwood, OH, 2009.

\end{thebibliography}
\end{document}